\documentclass[11pt]{article}

\usepackage[english]{babel}
\usepackage[usenames]{color}
\usepackage{amsfonts}
\usepackage{amsthm}
\usepackage{graphicx}
\usepackage{mathrsfs}
\usepackage{amsmath}
\usepackage{bm}
\usepackage{epstopdf}
\usepackage{subcaption}
\usepackage{url}

\usepackage{todonotes}

\DeclareGraphicsExtensions{.ps}

\theoremstyle{plain}
\newtheorem{df}{Definition}

\newtheorem{proposition}{Proposition}

\newtheorem{theorem}{Theorem}

\newtheorem{lemma}{Lemma}
\newtheorem{fact}[df]{Fact}

\newtheorem{corollary}[df]{Corollary}

\newtheorem{remark}{Remark}

\def\tk{\tilde{k}}

\renewcommand{\thefootnote}{\fnsymbol{footnote}}
\newcommand{\bea}{\begin{eqnarray}}
\newcommand{\eea}{\end{eqnarray}}
\newcommand{\be}{\begin{equation}}
\newcommand{\ee}{\end{equation}}
\newcommand{\beas}{\begin{eqnarray*}}
\newcommand{\eeas}{\end{eqnarray*}}

\newcommand{\bc}{\begin{center}}
\newcommand{\ec}{\end{center}}

\def\dag{{\mathrm{DAG}}}
\def\udag{{\mathrm{UDAG}}}
\def\Sy{ \mathbb{S}}

\def\R{ {\mathbb{R}}}

\def\intersect{\bigcap}
\def\union{\bigcup}
\def\lintersect{\cap}
\def\lunion{\cup}

\def\parsec{\par\noindent}
\def\med{\medskip\parsec}
\def\eod{\vrule height 6pt width 5pt depth 0pt}
\def\qed{\hfill$\eod$}

\def\E{ {\mathbb{E}}}

\newcommand{\fa}{{\mathfrak{A}}}
\newcommand{\fb}{{\mathfrak{B}}}

\newcommand{\G}{\mathbb{G}}
\newcommand{\PA}{\mathcal{PA}}

\newcommand{\dddeg}{{\underline{\underline{\deg}}}}
\newcommand{\ddeg}{{\underline{\deg}}}
\newcommand{\dd}{{\underline{d}}}
\newcommand{\ddd}{{\underline{\underline{d}}}}

\newcommand{\Binomial}{\mathrm{Binomial}}

\def\whp{{\textrm{whp}}}

\def\Pr{{P}}

\newcommand{\Aut}{\mathrm{Aut}}
\newcommand{\distrib}{\sim}

\newcommand{\Adm}{\mathrm{Adm}}
\newcommand{\Gr}{\mathcal{G}}

\begin{document}
\begin{center}
{\LARGE {\bf Asymmetry and Structural Information in Preferential Attachment Graphs}}\footnotemark[1]
\medskip
\parsec
\today

\medskip\bigskip
\begin{tabular}{lll}
Tomasz \L{}uczak&Abram Magner&Wojciech Szpankowski\\
Faculty of Math. \& Comp. Sci.&Center for the Science of Information&Dept.Computer Science\\
Adam Mickiewicz University&Purdue University&Purdue University\\
61-614 Pozna\'n&W. Lafayette, IN 47907&W. Lafayette, IN 47907\\
Poland&U.S.A.&U.S.A.\\
{\tt tomasz@amu.edu.pl}& {\tt abram10@gmail.com}& {\tt spa@cs.purdue.edu}
\end{tabular}

\footnotetext[1]{
This work was supported by
NSF Center for Science of Information (CSoI) Grant CCF-0939370, 
and in addition by NSF Grants CCF-1524312,  and
NIH Grant 1U01CA198941-01, and NCN grants 2012/06/A/ST1/00261 and 2013/09/B/ST6/02258.}

\end{center}
\med

\begin{abstract}

Graph symmetries intervene in diverse applications, 
from enumeration, to graph structure compression, to 
the discovery of graph dynamics (e.g., node arrival order inference). 
Whereas Erd\H{o}s-R\'enyi graphs are typically asymmetric,
real networks are highly symmetric.  So a natural question 
is whether preferential
attachment graphs, where in each step a new node with $m$ edges is added,
exhibit any symmetry.  
In recent work it was proved that preferential attachment graphs
are symmetric for $m=1$, and there is some 
non-negligible probability of symmetry for $m=2$. 
It was conjectured that these graphs are asymmetric when $m \geq 3$.  
We settle this conjecture
in the affirmative, then use it to estimate the 
structural entropy of the model.  To do this, 
we also give bounds
on the number of ways that
the given graph structure could have arisen by preferential
attachment.  These
results have further implications for information theoretic 
problems of interest on preferential attachment 
graphs.

\end{abstract}
\med\med

\newpage
\renewcommand{\thefootnote}{\arabic{footnote}}
\setcounter{footnote}{0}

\section{Introduction}

Study of the asymptotic behavior of the symmetries of graphs,
originally motivated by enumerative combinatorial problems, has recently
found diverse applications in problems ranging from graph compression to 
discovering interesting motifs to understanding dynamics of growing graphs.

Let us explore some of these applications in more detail. 
The basic problem of structural (unlabeled graph) compression can be formulated as
follows: given a probability distribution
on labeled graphs, determine a binary encoding of samples from
the induced 
unlabeled graph distribution so as to minimize expected description length.
In \cite{choiszpa2009} the authors studied this problem in the setting of
Erd\H{o}s-R\'{e}nyi graphs. They showed that, under any distribution
giving equal probability to isomorphic graphs,
the entropy of the induced
distribution on graph structures (i.e., isomorphism classes of graphs) is less
than the entropy of the original distribution by an amount proportional to the
expected logarithm of the number of automorphisms.  
Thus the solution to the above problem is intimately connected with the 
symmetries of the random graph model under consideration.
We explore this topic in some detail in Lemma~\ref{StructuralEntropyProposition} 
of Section~\ref{sec-main}.

We mention also a few representative algorithmic motivations for the study of graph symmetries.
The first involves the problem of \emph{motif discovery}: given a graph $G$ and
a pattern graph $H$, the problem is to find all subgraphs of $G$ that
are isomorphic to $H$.  It has been observed (see, e.g., \cite{picard}) that taking
into account the symmetries of $H$ can significantly decrease the time and space complexity
of the task.  The same is true for $G$ if it has nontrivial symmetries.  

In the area of Markov chains, the paper \cite{diaconis} studies the following problem:
given a graph $G$, the task is to assign weights to edges of $G$ so as to minimize the mixing
time of the resulting Markov chain.  The authors show that symmetries in $G$ may be exploited
to significantly reduce the size of a semidefinite program formulation which solves the problem.
Moreover, they point out several references to the literature in which symmetry plays a key role
in reducing complexity for various problems.

Study of symmetries is further motivated by their connection to 
various
measures of information contained in a graph structure.  For 
instance,
the \emph{topological entropy} of a graph, studied in 
\cite{rashevsky1955}
and \cite{trucco1956}, is a function on graphs that measures the 
uncertainty in the orbit class (i.e., the
set of nodes having the same long-term neighborhood structure) of 
a node
chosen uniformly at random from the node set of the graph.  Note 
that, unlike the labeled and unlabeled graph entropies that we 
study throughout this paper, the topological entropy is a 
function of a particular graph, rather than of a probability 
distribution on graphs.  If the graph
is asymmetric, then the topological entropy is
maximized: if $n$ is the size of the graph, then the topological 
entropy
is, to leading order, $\log n$.  In general, if the symmetries of 
the
graph can be characterized precisely, then so can the topological 
entropy.

The present paper is a step in the direction of understanding symmetries of complex networks
and toward extending graph structure compression results to random 
graph models other than Erd\H{o}s-R\'enyi. In particular, many
real-world graphs exhibit a power-law degree distribution (see \cite{durrett2006}).
A commonly studied model for real-world networks is the 
\emph{preferential attachment} mechanism introduced
in \cite{barabasialbert2002}, in which a graph is built
one vertex at a time, and each new vertex $t$ makes $m$ choices of neighbors in the current graph,
where it attaches to a given old vertex		
$v$ with probability proportional to the current degree of $v$.
We study here a simple variant of the preferential attachment model 
(see \cite{durrett2006} and the conclusion section for other models),
and in the conclusion of this paper 
we suggest that the symmetry behavior of other preferential models
can be studied using the approach developed here.
Our main result is the following: for the variant of the preferential 
attachment model under consideration, when each vertex added to the graph 
chooses a fixed number $m \geq 3$ neighbors, with high probability, there are no nontrivial
symmetries.  This is in stark contrast to the 
many symmetries observed in real-world networks 
\cite{macarthur2006}.  As we remark below our statement of
Theorem~\ref{th-main}, the asymmetry threshold (as well as the 
degree sequence power law exponent) changes with
other parameters in variants of preferential attachment, such 
as affine and nonlinear models, which may explain this 
discrepancy.

The problem of establishing asymmetry in preferential
attachment graphs appears to be difficult, and
literature on it seems to be scarce.
We are aware only of \cite{jms14}, which
proved that such graphs are {\it symmetric} for $m=1$ and (with asymptotically positive probability) also for $m=2$. 
The authors of \cite{jms14}
conjectured that preferential attachment graphs are indeed asymmetric for
$m\geq 3$. 
In this paper we first settle this conjecture in the affirmative using  different
methods than the one applied in \cite{kimvusudakov2002} and \cite{jms14}.
Namely, instead of relying on the graph defect (a measure of asymmetry defined in \cite{kimvusudakov2002} which is necessarily bounded in
preferential attachment graphs, and hence has poor concentration properties),
we shall observe that symmetry would 
imply that certain vertices make the same choices with regard to an 
initial set of vertices uniquely identifiable by their degrees, which we 
prove is unlikely to happen for		
preferential attachment graphs whenever $m \geq 3$.

After settling the asymmetry question for preferential attachment graphs, we use
it to address the issue of structural entropy. We first review 
an estimate of
the labeled graph entropy given in \cite{sauerhoff},
and then estimate the unlabeled graph entropy (also known as the structural entropy). In 
Lemma~\ref{StructuralEntropyProposition} we
relate both entropies. Then, using 
our asymmetry result from Theorem~\ref{th-main}, we
estimate the structural entropy.  To derive the structural entropy estimate,
we study the characteristics of the directed, acyclic graph version of
the preferential attachment process (culminating in 
Proposition~\ref{FatDAGProposition}, which may be of independent interest).  In particular, we estimate the number of ways that
a given graph could have arisen according to the preferential
attachment mechanism.  We additionally estimate the typical \emph{height} (i.e., the length of the longest directed
path) of this directed version of the graph, which, being a natural structural quantity, may be of independent interest.  

We emphasize that the labeled and unlabeled graph entropies that
we study are fundamental, as they give the minimum achievable 
expected length of any prefix source code (i.e., compression code) for 
these graphs.

Now we review some of the literature on symmetries of random graphs.
The study of the asymptotic behavior of the automorphism group of a random graph
started with a paper of  Erd\H{o}s and R\'{e}nyi~\cite{erdosrenyi1963}, where
they showed that $\G(n,M)$ (i.e., the uniformly random graph on $n$ vertices with $M$ edges) with
constant density (i.e. when $M=\Theta (n^2)$) is asymmetric with high probability,
a result motivated by the combinatorial question of determining the asymptotics
of the number of unlabeled graphs on $n$ vertices for $n \to \infty$. Then Wright~\cite{wright1974} 
proved that $\G(n,M)$ whp becomes asymmetric as soon as 
the number of isolated vertices in it drops under $1$. His result was later strengthened 
by Bollob\'as~\cite{bollobas1982Dist}, who also proved asymmetry  for $r$-regular graphs with $r\ge 3$. The asymptotic size of the automorphism group of $\G(n,M)$ for small $M$,  
where the graph is not connected, was given  by {\L}uczak~\cite{luczak1988}. 
As a similar question motivated the investigation of symmetry properties of random
regular graphs, Bollob\'as improved his result from~\cite{bollobas1982Dist}  by showing in~\cite{bollobas1982} 
that unlabelled  regular graphs with degree $r\ge 3$ 
are whp asymmetric as well. Let us note that it is a substantially stronger theorem (see the discussion below 
after Theorem~\ref{th-main}).

For general models, asymmetry results can be nontrivial to prove, due in
part to the fact that asymmetry is a global property.  Furthermore, the
particular models considered here present difficulties not seen in
the Erd\H{o}s-R\'{e}nyi case: there is significant dependence between edge
events, and graph sparseness makes derivation of concentration results
difficult. 
However, settling the symmetry/asymmetry question opens the door to several
other lines of investigation, including, e.g., the design of optimal 
structural compression schemes and the
precise characterization of the limits of inference problems (see, for 
example, \cite{nodeage}) for preferential attachment graphs.
These applications crucially depend on our precise 
understanding of graph symmetry.

The paper is organized as follows. In Section~\ref{sec-main} we
present our main results regarding the graph asymmetry and structural entropy.
In Section~\ref{DegreeSequenceSection}, we state and prove several results
on the degree sequence which will be useful in subsequent proofs.
We prove the graph asymmetry result in Section~\ref{sec-proof1} and the entropy 
result, along with the necessary structural results on the directed version of the model, in Section 
~\ref{StructuralEntropyTheoremProofSection}.

\section{Main Results}
\label{sec-main}

In this section, we state the main problem, introduce the model that we
consider, and formulate the main results.
First, we review some standard graph-theoretic terminology and notation. 

We start with the notion of structure-preserving
transformations between labeled graphs: given two graphs (possibly with multiple edges between nodes)
$G_1$ and $G_2$ with
vertex sets $V(G_1)$ and $V(G_2)$, a mapping $\phi:V(G_1)\to V(G_2)$ is said
to be an \emph{isomorphism} if it is bijective and preserves edge relations;
that is, for any $x, y\in V(G_1)$, the number of edges (possibly $0$) between $x$ and $y$ is equal to the
number of edges in $G_2$ between $\phi(x)$ and $\phi(y)$.  When
such a $\phi$ exists, $G_1$ and $G_2$ are said to be isomorphic.

An isomorphism from a graph $G$ to itself is called an
\emph{automorphism} or \emph{symmetry} of $G$.  The set of automorphisms of $G$,
together with the operation of function composition, forms a group, which
is called the \emph{automorphism group} of $G$, denoted by $\Aut(G)$. 
We then say that $G$ is \emph{symmetric} if it has at least one {\it nontrivial
symmetry} and that $G$ is \emph{asymmetric} if the only symmetry of $G$ is
the identity permutation.

Our first main goal is to answer, for $G$ distributed according to a 
\emph{preferential attachment} model (defined below),
the question of whether with high probability 
the automorphism group is trivial (i.e., $|\Aut(G)|=1$) or not.

We say that a multigraph $G$ on vertex set $[n]=\{1,2,\dots, n\}$ 
is {\em $m$-left regular} 
if the only loop of $G$ is at the vertex $1$, and each vertex $v$, 
$2\le v\le n$, has precisely
$m$ neighbors in the set $[v-1]$. We will study a variant of the 
{\em preferential attachment model} 
by giving a probability 
measure $\PA(m; n)$ on 
the set of all $m$-left regular graphs on $n$ vertices.  
More precisely, for an integer  parameter $m\geq 1$ we define the graph $\PA(m; n)$ with vertex set 
$[n]=\{1,2,\dots,n\}$ using recursion on $n$  in the following way: 
the graph $G_1 \distrib \PA(m;1)$ is a single node with label $1$ with $m$ self-edges (these will be
the only self-edges in the graph, and we will only count each such edge once in the
degree of vertex $1$).
Inductively, to obtain a graph $G_{n+1} \distrib \PA(m;n+1)$ from $G_{n}$,
we add vertex $n+1$ and make $m$ random choices $v_1, ..., v_m$ of neighbors in $G_n$ as follows:
for each vertex $w \leq n$ (i.e., vertices in $G_n$),
\begin{align*}
	P( v_i = w  | G_n, v_1, ..., v_{i-1})
    = \frac{ \deg_n(w) }{ 2mn },
\end{align*}
where throughout the paper we denote by $\deg_n(w)$ the degree of 
vertex $w \in [n]$ in the graph $G_n$ (in other words, the degree 
of $w$ after vertex $n$ has made all of its choices).

As stated above, the model so defined is a variant of preferential 
attachment, which refers to a class of similar dynamic random 
graph models in which new vertices choose $m$ neighbors with 
probability proportional to their current degrees.  The 
introduction of this notion is often credited to
\cite{barabasialbert2002}.  Several different tweaks have been
considered: e.g., whether or not self-loops are allowed, whether
or not degrees are updated after each choice made by a given 
vertex, etc.  Our results are robust to such tweaks (in 
particular, all of our main theorems hold as stated, regardless
of the variations mentioned above).  The proofs are similarly 
robust; however, the proof of our
structural entropy result becomes somewhat more involved if,
for example, multiple edges are replaced with single ones 
(yielding a simple graph).  For this reason, we focus on the
model defined above.

We next formulate our first main result regarding asymmetry of 
$\PA(m;n)$ for $m\geq 3$
that we prove in Section~\ref{sec-proof1}.

\begin{theorem}[Asymmetry for preferential attachment model]
\label{th-main}
    Let $G\distrib \PA(m; n)$ for fixed $m \geq 3$.  Then, with high probability as $n\to\infty$,
    \begin{align*}
        |\Aut(G)| = 1.
    \end{align*}
	More precisely,  for $m \geq 3$,
	\begin{align}
	\label{eq-main}
		P(|\Aut(G)| > 1) 
        = O(n^{-\delta}),
	\end{align}
    for some fixed $\delta > 0$ and large $n$.
\end{theorem}

\begin{remark}
	One may wonder if one can strengthen the above statement and 
    claim that 
	for $m\ge 3$ we have $\E|\Aut(\PA(m;n))|=1+o(1)$; if this 
    would be the case, 
	then a natural unlabeled version of the model, which we 
    denote by $\PA^u(m; n)$, defined below would be with high 
    probability asymmetric too.
	However, somewhat surprisingly, it is not the case. 
	
	To make this precise,
	let us recall that in the case of the uniform random graph model $\G(n,M)$, 
	where we choose a graph uniformly at random from the family of all graphs with 
	$n$ labeled vertices and $M$ edges, the automorphism group becomes with high probability trivial 
	just above the connectivity threshold; i.e., when $2M/n-\log n\to \infty$;
	in fact, at this moment the expected size of $\Aut(\G(n,M))$ is 
	$1+o(1)$. Moreover,  
	almost precisely at this time the unlabeled uniform random graph $\G^u(n,M)$ 
	which is chosen at random from the family of all unlabeled graphs with 
	$n$  vertices and $M$ edges becomes asymmetric and, furthermore, 
	the structure of  $\G^u(n,M)$ is almost identical to the structure of $\G(n,M)$; i.e.,
	$\G^u(n,M)$ is basically $\G(n,M)$ with erased labels (for more information on this model, see \cite{luczak1991}). 
    As we have already mentioned above, the same is true for $r$-regular random graphs with $r\ge 3$, 
where the uniform labeled and unlabeled graph models have basically the same asymptotic properties~ \cite{bollobas1982}.

	Returning to the preferential attachment case,
	for any $m$-left regular graph $G$ let $S(G)$ denote the class of all $m$-left regular graphs 
	which are isomorphic to $G$, and let $\mathcal{S}$ denote the family 
	containing all $S(G)$, i.e. 
	the family of all `unlabeled $m$-left regular graphs'. We define the 
	unlabeled graph distribution $\PA^u(n;m)$ 
	as the probability distribution on $\mathcal{S}$, where the probability of 
	each class $S(G)$  
	is proportional to the average of the probabilities that a labeled version 
	of $S(G)$ is $\PA(m;n)$, i.e. proportional to 
	$$\frac{1}{|S(G)|}\sum_{H\in S(G)}P(H=\PA(m;n))\,.$$
	Note that this is a different distribution from the one that samples a 
	preferential attachment graph and takes its isomorphism class.
	Note, also, that $\G^u(n,M)$ is defined in the same way, but in that case all 
	terms in the sum are the same, so each equivalence class is equally likely. 
	Now we can ask whether the typical structure of $\PA^u(m; n)$ 
	is the same, or very close to that of $\PA(m; n)$ (i.e., a 
    preferential attachment graph with labels removed); in 
    particular whether it is asymmetric. It seems that it is not 
    this case. To see this, notice that the typical $\PA(m; n)$ is 
    asymmetric and, furthermore, whp it contains
	$L=\Omega_m(n)$ vertices with label at least $3n/4$, such that they 
    are of degree $m$ and have pairwise different neighborhoods 
    contained  in $[n/2]=\{1,2,\dots,n/2\}$.   For such a graph G we 
    clearly have $|S(G)|\ge  L! = \exp(\Omega_m(n \log n))$ and so
    for every $H\in S(G)$ we have 
	$\Pr( H=\PA(m;n))\le \exp(-\Omega_m(n\log n))$ (if the 
    neighborhoods are not different, it is hard to get such an $S(G)$). 
    On the other hand for the graph $H'$ such that 
	all vertices of labels $\ell\ge m+1$ has neighbors $\{1,2,\dots,m\}$ we have 
	$\Pr(H'= \PA(m; n))\ge \exp(-O_m(n))$. Thus, the very asymmetric 
	 $H'$ is much more likely to appear as $\PA^u(m; n)$ than a `typical' graph 
	 $\PA(m; n)$.
	 
	Here we will not investigate the properties of $\PA^u(m; n)$ but rather 
	characterize the information content of the distribution on unlabeled
	graphs given by sampling from
	$\PA(m; n)$ and removing the labels (i.e., taking the isomorphism class
	of the sampled graph).
\end{remark}

\begin{remark}
	The particular threshold at $m=3$ differs in different
    variations of the model.  For example, in \emph{affine}
    preferential attachment, where at time $t$ a vertex $v$
    is chosen with conditional probability proportional to
    $\deg_{t-1}(v) + \delta$, for a fixed parameter 
    $\delta > -m$ (and an initial graph with minimum degree
    greater than $-\delta$),
    our proof technique shows that we have asymmetry with
    high probability whenever
    \begin{align}
    	m > \frac{\sqrt{\delta^2 + 4} - \delta + 2}{2}.
        \label{AsymmetryThreshold}
    \end{align}
    Note that when $\delta = 0$, this gives $m > 2$, in agreement
    with our theorem statement.  When $\delta \to \infty$,
    this translates to $m \geq 2$ 
    (i.e., the asymmetry threshold decreases).  When $\delta < 0$, this places a restriction on $m$ in order for the model
    to be well-defined: $m > -\delta$.  Thus, the set of
    values of $m$ for which our proof technique does not show
    asymmetry is of cardinality $\Theta(1)$.  For instance,
    when $\delta = -2$, it is required that $m > 2$, and
    (\ref{AsymmetryThreshold}) becomes $m \geq 4$.  
\end{remark}


As a direct application of Theorem~\ref{th-main}, we estimate the
structural entropy $H(S(G))$. 
Recall that the entropy $H(G)$ of the labeled graph $G\sim \PA(m; n)$ is defined as
\begin{align*}
	H(G)=-\sum_{G\in \Gr_n} P(G) \log P(G),
\end{align*}
where $\Gr_n$ denotes the set of graphs on $n$ vertices.
The structural entropy $H(S(G))$ is then simply the entropy of the 
isomorphism type of $G$.  Note that these are Shannon entropies,
which are functionals of probability distributions, rather than
of fixed graphs.
We next show how to find a relation between these two entropies.  By
the chain rule for conditional entropy, 
\begin{align}
	H(G)
    = H(S(G)) + H(G | S(G)).
\end{align}
The second term, $H(G | S(G))$, measures our uncertainty about the 
labeled graph if we are given its structure.
We will give a formula for $H(G | S(G))$ in terms of $|\Aut(G)|$ and
another quantity, defined as follows: suppose that, after generating
$G$, we relabel $G$ by drawing a permutation $\pi$ uniformly at random from 
$\Sy_n$, the symmetric group on $n$ letters, and computing $\pi(G)$.  Then
conditioning on $\pi(G)$ yields a probability distribution for possible
values of $\pi^{-1} = \sigma$. We can write $H(G | S(G))$ in
terms of $H(\sigma | \sigma^{-1}(G)) = H(\sigma | \sigma(G))$ and $\E[\log |\Aut(G)|]$ using the chain rule for entropy,
resulting in the following lemma.

\begin{lemma} [Structural entropy for preferential attachment graphs]
        \label{StructuralEntropyProposition}
        Let $G \distrib \PA(m; n)$ for fixed $m \geq 1$, and let $\sigma$
        be a uniformly random permutation from $\Sy_n$.  Then we have
    \begin{align}
        H(G) - H(S(G))
        = H(\sigma | \sigma(G)) - \E[\log |\Aut(G)|].
        \label{EntropyDifference}
    \end{align}
\end{lemma}
\begin{proof}
	We compute $H(G, \sigma(G), \sigma | S(G))$ in two different ways, by the chain rule.
    \begin{align*}
    	H(G, \sigma(G), \sigma | S(G))
        &= H(G | S(G)) + H(\sigma(G) | G) + H(\sigma | G, \sigma(G), S(G)) \\
        &= H(G | S(G)) + H(\sigma(G) | G) + H(\sigma | G, \sigma(G)),
    \end{align*}
    where the second equality is because knowing $G$ implies that we know $S(G)$ (so that conditioning on both
    $G$ and $S(G)$ is equivalent to conditioning on just $G$).
    Since $\sigma$ is chosen uniformly at random, it is a uniformly random isomorphism between $G$ and
    $\sigma(G)$.  There are always exactly $|\Aut(G)|$ such isomorphisms, so we have that the third term
    is $H(\sigma | G, \sigma(G)) = \E[\log |\Aut(G)|]$.  Thus, our first expression is
    \begin{align}
    	H(G, \sigma(G), \sigma | S(G))
        = H(G | S(G)) + H(\sigma(G) | G) + \E[\log |\Aut(G)|].
        \label{FirstExpr}
    \end{align}
    
    We evaluate the expression alternatively as
    \begin{align}
    	H(G, \sigma(G), \sigma | S(G))
        &= H(\sigma(G) | S(G)) + H(\sigma | \sigma(G), S(G)) + H(G | \sigma, \sigma(G), S(G)) \nonumber \\
        &= H(\sigma(G) | S(G)) + H(\sigma | \sigma(G)).
        \label{SecondExpr}
    \end{align}
    Here, the last term of the second expression is $0$ because $G$ is a deterministic function of the pair $(\sigma, \sigma(G))$.
    The second term in the same expression simplifies because knowing $\sigma(G)$ implies that we also know $S(G)$.
    
    Combining (\ref{FirstExpr}) and (\ref{SecondExpr}) yields
    \begin{align*}
    	H(G | S(G)) = H(\sigma | \sigma(G)) - \E[\log |\Aut(G)|] + H(\sigma(G) | S(G)) - H(\sigma(G) | G).
    \end{align*}
    Now, we claim that $H(\sigma(G) | S(G)) = H(\sigma(G) | G)$, which will complete the proof.  This is a result
    of the fact that, under both conditionings, $\sigma(G)$ is a uniformly random element of $S(G)$, so that both
    conditional entropies must be equal.  We thus have
    \begin{align*}
    	H(G | S(G)) = H(\sigma | \sigma(G)) - \E[\log |\Aut(G)|],
    \end{align*}
    as desired.
\end{proof}

\begin{remark}
	In the proof of Theorem~\ref{StructuralEntropyTheorem} below,
    we prove an alternative, more combinatorial representation for
    $H(\sigma | \sigma(G))$; see (\ref{HCondSigAlternate}).
\end{remark}

To estimate the structural entropy $H(S(G))$ using Lemma~\ref{StructuralEntropyProposition}, we need 
an expression for the labeled graph entropy $H(G)$
and to evaluate the two terms on the right-hand side of
(\ref{EntropyDifference}). 

A one-term expansion for the labeled entropy of the preferential attachment model
is given in Theorem~1 of \cite{sauerhoff}:
\begin{align*}
	H(G) = m n\log n + \Theta(n).
\end{align*}
	Using refined results on the degree sequence, we are able to give a precise formula for
    the second term, resulting in an error term of $o(n)$.  In particular, the $\Theta(n)$ term is
    \begin{align*}
    	m\left(\log 2m - 1 - \log m! - A\right) n + o(n),
    \end{align*}
    where
    \begin{align*}
    	A(m) = A = \sum_{d=m}^{\infty} \frac{\log d}{(d+1)(d+2)}. 
    \end{align*}
    However, for the present application,
    this level of precision is not needed.

Now, we are in the position to complete our computation of the structural entropy.
\begin{theorem}[Structural entropy of preferential attachment graphs]
	\label{StructuralEntropyTheorem}
	Let $m \geq 3$ be fixed.
    Consider $G \distrib \PA(m; n)$.  We have
    \begin{align}
    	H(S(G))
        = (m-1) n\log n + R(n),
    \end{align}
    where $R(n)$ satisfies
    \begin{align*}
        C n
        \leq |R(n)| 
        \leq O(n\log \log n)
    \end{align*}
    for some nonzero constant $C=C(m)$.
\end{theorem}
To prove this, we evaluate (\ref{EntropyDifference}) by relating $H(\sigma | \sigma(G))$ to a combinatorial
parameter of the directed version of $G$.
We show this derivation in Section~\ref{StructuralEntropyTheoremProofSection}.

\section{Results on the Degree Sequence}
\label{DegreeSequenceSection}
In this section, we present results on the degree sequence of preferential attachment
graphs which we will use in the proofs of our main results in subsequent sections.

First, we denote by $\deg_t(s)$ the degree of a vertex $s < t$
after time $t$ (i.e., after vertex $t$ has made its choices).

The first result, which specializes Lemma~4 of \cite{BollobasRiordan2004}, 
gives an upper bound on the probability that two given vertices
are adjacent.
\begin{lemma}
	\label{AdjProbBoundCorollary}
    Let $w < v$.  Then the probability that $v$ is adjacent to $w$ is bounded above by $\frac{Cm}{\sqrt{vw}}$.
    In particular, each two vertices $v, w \geq \epsilon n$
    are adjacent with probability smaller than $\frac{Cm}{\epsilon n}$.
\end{lemma}

It is known that for $t>s$, the expectation of $\deg_t(s)$ is $O(\sqrt{t/s})$ (see \cite{hofstad}, Theorem 8.2). 
We first state a simple tail bound to the right of this expectation, which
may be found in 
\cite{FriezeKaronski}, Lemma 17.2:  
\begin{lemma}[Right tail bound for a vertex degree at a specific time]
        \label{FriezeKaronskiLemma}
    Let $r < t$.  Then
    \begin{align*}
        P(\deg_{t}(r) \geq Ae^m (t/r)^{1/2}(\log t)^2)
        = O(t^{-A})
    \end{align*}
for any constant $A>0$ and any $t$.
\end{lemma}

We can prove a similar left tail bound for the random variable $\deg_{t}(s)$ whenever
$s \ll t$, as captured in the following lemma.

\begin{lemma}[Degree left tail bound]
	Let $v = O(T^{1-\epsilon})$ as $T \to \infty$, for some fixed $\epsilon \in (0, 1/2)$.
    Then there exist some $C, D > 0$ such that
    \begin{align}
    	\Pr\left( \deg_{T}(v) < C \left(\frac{T}{v} \right)^{(1-\epsilon)^2/(2.0001)} \right)  
        \leq 
        e^{-D\epsilon^3 \log(T)}
        = T^{-D \epsilon^3}.
        \label{LeftTailBoundProbExpr}
    \end{align}
    \label{DegreeLeftTailLemma}
\end{lemma}
To prove this, we need the following coarser lemma.
\begin{lemma}
	\label{LogLeftTailBoundLemma}
    Let $v < T^{1-\epsilon}$, for some fixed $\epsilon > 0$.  Then there exist constants $C, D > 0$ independent of $\epsilon$
    such that
    \begin{align}
    	\Pr(\deg_{vT^{\epsilon}}(v) < C\epsilon \log T)
        \leq T^{-D\epsilon}
    \end{align}
    for $T$ sufficiently large.
\end{lemma}
\begin{proof}
    We observe the graph at exponentially increasing time steps: for some $\beta > 0$,
    let $t_0 = v$, $t_j = (1+\beta)^j t_0$, $t_k = (1+\beta)^k t_0 = vT^{\epsilon}$ 
    (so $k = \frac{\epsilon \log T}{\log(1+\beta)}$).
    Note that $\deg_{t_0}(v) = \deg_{v}(v) = m$.  
    
    Let us upper bound the probability $p_{j+1}$ that no connection to vertex $v$ is
    made by any vertex in the subinterval $(t_{j}, t_{j+1}]$:
    \begin{align}
    	p_{j+1}
        \leq \left( 1 - \frac{m}{2mt_{j+1}} \right)^{m(t_{j+1} - t_j)}
        = \left( 1 - \frac{1}{2t_{j+1}} \right)^{m\beta t_j},
    \end{align}
    which is at most some positive constant $\rho = \rho(m\beta)$, uniform in $j$,
    satisfying $\rho < 1$.  This follows from the inequality $1 - x \leq e^{-x}$ for all
    $x \in \R$.
    Thus, the total number of connections to vertex $v$ in all subintervals can be
    stochastically lower bounded by a binomial random variable with number of trials $k = \Theta(\epsilon \log T)$ and success probability $\rho(m\beta)$: for any $d \geq 0$,
    \begin{align}
    	\Pr(\deg_{t_k}(v) - m \geq d)
        \geq \Pr(\Binomial(k, 1-\rho) \geq d).
    \end{align}
    In particular, as $T \to \infty$, this implies (using the Chernoff bound) that
    with probability $1 - T^{-D\epsilon}$, the number of subintervals which contribute
    at least one new edge to $v$ is at least $C\epsilon \log T$, for some $C$, so
    that $\deg_{vT^{\epsilon}}(v) \geq C\epsilon \log T$, which completes the proof. 
\end{proof}

With the previous lemma in hand, we are now ready to prove our left tail bound.
\begin{proof}[Proof of Lemma~\ref{DegreeLeftTailLemma}]
	Similar to the proof of Lemma~\ref{LogLeftTailBoundLemma}, we observe the graph
    at exponentially increasing times: fix a small $\alpha > 0$, and let
    $t_0 = vT^{\epsilon}$, $t_j = (1+\alpha)^j t_0$, $t_k = (1+\alpha)^k t_0 = T$, 
    so that $k = \frac{\log(T/t_0)}{\log(1+\alpha)}$.  Denote by 
    $d_j = \deg_{t_j}(v)$ and $\Delta_{j+1} = d_{j+1} - d_j$, for each $j$.
    
	In the interval $(t_j, t_{j+1}]$, conditioned on the graph up to time $t_j$, $\Delta_{j+1}$ is
	stochastically lower bounded by a binomially distributed
	random variable with parameters $(t_{j+1} - t_j)m = \alpha t_j m$ and $p_{j+1} = \frac{d_j}{2mt_{j+1}}$.
    The former parameter is simply the interval length (in terms of number of vertex choices).
	The latter parameter comes from the fact that the degree of $v$ at any point in
    the interval is at least $d_j$, and the total degree of the graph is at most $2mt_{j+1}$.
	I.e.,
	\begin{align}
		\Delta_{j+1}\big| G_{t_j} \succeq_{st} \Binomial\left(m\alpha t_j, \frac{d_j}{2 t_{j+1} m} \right),
        \label{StochDomExpr}
	\end{align}
    where $\succeq_{st}$ denotes stochastic domination.

    This suggests that we define the \emph{bad} event $B_{j} = [\Delta_{j} <  \alpha t_{j-1} m p_{j}(1-\epsilon) ]$, for arbitrary $\epsilon > 0$,
	and for $j \in [1, k]$.  We further define $B_0 = [d_0 < C\epsilon\log T]$, 
    for some constant $C > 0$.
    
	Conditioning on all of the $B_j$ (for $j \in \{0, ..., k\}$) failing to hold, we have
	\begin{align}
		\Pr\left(\intersect_{j < k} \left[ d_{j+1} \geq d_{j} \left( 1 + \frac{(1-\epsilon)\alpha}{2(1+\alpha)} \right)\right] ~\big| ~ \intersect_{j=0}^{k} \neg B_j \right)
	    = 1,
        \label{CondEqn}
	\end{align}
	recalling that $d_{j+1} = d_j + \Delta_{j+1}$ by definition.
	This in particular implies that (still under the same conditioning)
	\begin{align}
		d_k \geq d_0 \cdot \left( 1 + \frac{(1-\epsilon)\alpha}{2(1+\alpha)} \right)^{k}
	    = d_0 \exp \left( \log(T/t_0) \frac{\log(1 + \frac{(1-\epsilon)\alpha}{2(1+\alpha)} )}{\log(1 +\alpha) } \right).
	\end{align}
	Taking $\alpha$ close enough to $0$, this becomes
	\begin{align}
		d_k \geq 
		d_0 \exp\left( \frac{1-\epsilon}{2(1 + o_{\alpha\to 0}(1))}\log(T/t_0) \right)
	    = d_0 (T/t_0)^{\frac{1-\epsilon}{2.0001}},
	    \label{DegreeLowerBoundEqn}
	\end{align}
    as in the statement of the lemma.
    
    Now, it remains to lower bound the probability $\Pr(\intersect_{j=0}^k \neg B_j)$.
    We may write it as
    \begin{align*}
    	\Pr\left(\intersect_{j=0}^k \neg B_j \right)
        = \Pr(\neg B_0) \prod_{j=1}^{k} \Pr(\neg B_j | \neg B_0, ..., \neg B_{j-1})
        \geq (1 - T^{-D\epsilon}) \prod_{j=1}^{k} \Pr(\neg B_j | \neg B_0, ..., \neg B_{j-1}),
    \end{align*}
    where the inequality is by Lemma~\ref{LogLeftTailBoundLemma}.

    Now, by the stochastic domination (\ref{StochDomExpr}), the conditioning,
    and the Chernoff bound, the $j$th factor of the product is lower bounded as
    follows:
	\begin{align}
		\Pr(\neg B_j | \neg B_0, ..., \neg B_{j-1})
	    &\geq \Pr( \Binomial(\alpha t_{j-1} m, p_{j}) \geq \alpha t_{j-1} m p_j (1-\epsilon) | \neg B_0, ..., \neg B_{j-1} ) \\
	    &\geq 1 - \exp\left( - \frac{\epsilon^2 \alpha d_{j-1}}{2(1+\alpha)} \right). \nonumber
	\end{align}
	Under the conditioning, $d_{j-1}$ is further lower bounded by $\left( 1 + \frac{(1-\epsilon)\alpha}{2(1+\alpha)} \right)^{j-1} C\epsilon \log T 
    \geq \left(1 + \frac{\alpha}{4(1+\alpha)} \right)^{j-1} C\epsilon \log T$ (using the fact that $\epsilon < 1/2$), resulting in
	\begin{align}
		\Pr(\neg B_j | \neg B_0, ..., \neg B_{j-1})
	    \geq 
	    1 - \exp\left( - C\frac{\epsilon^3 \alpha}{2(1+\alpha)} \cdot \left(1 + \frac{\alpha }{4(1+\alpha)} \right)^{j-1} \log(T) \right).
	\end{align}
	
	This implies
	\begin{align}
		\Pr\left(\intersect_{j=0}^k \neg B_j \right) 
	    \geq \Pr(\neg B_0) \cdot \prod_{j=1}^k \left(1 - \exp\left( - C\frac{\epsilon^3 \alpha}{2(1+\alpha)} \cdot \left(1 + \frac{\alpha }{4(1+\alpha)} \right)^{j-1} \log(T) \right) \right).
	    \label{ImportantProduct}
	\end{align}
	For convenience, set $C' = C\frac{\alpha}{2(1+\alpha)}/D'$ and $D' = 1 + \frac{\alpha}{4(1+\alpha)}$.  Note that $D' > 1$.  So the product in (\ref{ImportantProduct}) can be written (after some simple asymptotic analysis) as
	\begin{align*}
    	\prod_{j=1}^k \left(1 - \exp\left( -\epsilon^3 C' \cdot D'^{j} \log(T) \right) \right)
        = 1 - \Theta(T^{-\epsilon^3 C'D'}).
    \end{align*}
    This implies, after combination with the lower bound on $\Pr(\neg B_0)$, that
    we can write
    \begin{align}
    	\Pr\left(\intersect_{j=0}^k \neg B_j \right)
        \geq (1 - T^{-D\epsilon})(1 - \Theta(-T^{-\epsilon^3C'D'}))
        \geq 1 - T^{-D''\epsilon^3},
    \end{align}
    for some $D'' > 0$ (depending on $\alpha$), as claimed.  Combining this with (\ref{CondEqn}) yields (\ref{DegreeLowerBoundEqn}) with the claimed probability bound as follows:
    \begin{align*}
    	\Pr(d_k \geq d_0(T/t_0)^{\frac{1-\epsilon}{ 2.0001}})
        &\geq \Pr\left(\intersect_{j < k} \left[ d_{j+1} \geq d_{j} \left( 1 + \frac{(1-\epsilon)\alpha}{2(1+\alpha)} \right)\right] \right) \\
        &\geq \Pr\left(\intersect_{j < k} \left[ d_{j+1} \geq d_{j} \left( 1 + \frac{(1-\epsilon)\alpha}{2(1+\alpha)} \right)\right] ~\big| ~ \intersect_{j=0}^{k} \neg B_j \right) \cdot \Pr\left(\intersect_{j=0}^k \neg B_j \right) \\
        &\geq 1 \cdot (1 - T^{-D''\epsilon^3}),
    \end{align*}
    as required.
\end{proof}

Using Lemma~\ref{DegreeLeftTailLemma}, we can prove a corollary roughly lower bounding
the typical minimum degree of the collection of vertices before a given time.
\begin{corollary}
	Let $\Delta > 0$ be fixed. 
	There exists some small enough $\delta > 0$ and positive constant $D$ such that
    \begin{align}
    	\Pr\left( \union_{w < T^{\delta}} \deg_{T}(w) < C\left( T^{1-\Delta} \right)^{1/2} \right)
        \leq T^{-D}
    \end{align}
    as $T\to\infty$.
    \label{EarliestDegreesCorollary}
\end{corollary}
\begin{proof}
	This follows immediately from the fact that the probability bound in Lemma~\ref{DegreeLeftTailLemma}
    is monotone in $\epsilon$ and constant with respect to $v$.  We omit the simple details.
\end{proof}

The next result gives a bound on the probability that two early vertices
have the same degree.
\begin{lemma}
        \label{lkey}
        There exist positive constants $\Delta < 1$ and $c$ such that
	the probability that for some $s<s'<k^2=n^{2\Delta}$ 
	we have $\deg_n(s)=\deg_n(s')$ 
	is $O(n^{-c})$.  
\end{lemma}
\begin{proof}
Let $s<s'<k^2=n^{2\Delta}$, for some $\Delta > 0$ to be chosen.  
We first estimate the probability that $\deg_n(s)=\deg_n(s')$.
In order to do so we set $n'=n^{0.6}$ and define
$$\ddeg(s)= \deg_{n-n'}(s)\quad \textrm{ and }\quad  \dddeg(s)=\deg_n(s)- \ddeg(s)\,.$$
Note that
\begin{align}
P(\deg_n(s)=\deg_n(s'))=&\sum_{\dd,\dd',\ddd'}P(\deg_n(s)=\deg_n(s')| \ddeg(s)=\dd,\ddeg(s')=\dd',
\dddeg(s')=\ddd')\nonumber\\
&\quad\quad \times P(\ddeg(s)=\dd,\ddeg(s')=\dd',
\dddeg(s')=\ddd')\nonumber\\
=&\sum_{\dd,\dd',\ddd'}P(\dddeg(s)=\dd'+\ddd'-\dd| \ddeg(s)=\dd,\ddeg(s')=\dd',
\dddeg(s')=\ddd')\nonumber \\
&\quad\quad \times P(\ddeg(s)=\dd,\ddeg(s')=\dd',
\dddeg(s')=\ddd')\,.\label{eq110}
\end{align}

Observe that due to Lemma~\ref{DegreeLeftTailLemma} (alternatively, Corollary~\ref{EarliestDegreesCorollary}) and Lemma~\ref{FriezeKaronskiLemma}, with probability
$1-O(n^{-c})$, for some appropriate $c > 0$ and small enough $k=n^{\Delta}$, a vertex $s\in [k^2]$ has degree between $n^{0.488}$ and $n^{0.51}$ at any time in the interval $[n-n', n]$.  Importantly, note that if this holds with probability $1 - O(n^{-c})$
for a given choice of $\Delta$, then the same holds for all smaller choices of $\Delta$, with the same value for $c$ (this is a consequence of the fact that the probability bound in Lemma~\ref{DegreeLeftTailLemma} is a function of $\epsilon$ and not of $v$).

Furthermore, one can estimate the random variable
$\dddeg(s)$ conditioned on $\ddeg(s)=\dd$ from above and 
below by binomial distributed random variables and use Chernoff bound to show that 		
with probability at least $1-O(n^{-c})$ we have 
\begin{equation}\label{eq111}
	\Big|\frac{\dd n'}{2mn}-\dddeg(s)\Big|=\Big|0.5m^{-1}\dd n^{-0.4}-\dddeg(s)\Big|\le \Big(\frac{\dd n'}{2mn}\Big)^{0.6}\le n^{0.08}\,.
\end{equation}

Thus, in order to estimate $P(\deg_n(s)=\deg_n(s'))$, it is enough 
to bound 
\begin{align*}
\rho(\dd',\ddd',\dd)&=P(\dddeg(s)=\dd'+\ddd'-\dd| \ddeg(s)=\dd,\ddeg(s')=\dd',
\dddeg(s')=\ddd')
\end{align*}
for $n^{0.488}\le \dd,\dd'\le n^{0.51}$ and 
$$|0.5\dd n^{-0.4}/m - (\dd'+\ddd'-\dd)|\le n^{0.08}\,.$$
In order to  simplify the notation set $\ell=\dd'+\ddd'-\dd$. 
Let us estimate the probability that $\dddeg(s)=\ell$ conditioned on $\ddeg(s)=\dd$ 
and $\ddeg(s')=\dd'$. The probability that some vertex $v>n-n'$ is connected 		
to $s$ by more than one edge is bounded from above by 
$$Cn'\Big(\frac{m\deg_{n}(s)}{n-n'}\Big)^2  \le n^{0.6}\,O(n^{-0.98})=O(n^{-0.38})\,$$
so we can omit  this case in further analysis.
The probability that we connect a given vertex $v>n-n'$ with $s$ 
 is given by 
 \begin{equation}\label{eq112}
\frac{m\deg_{v-1}(s)}{2m(v-1)}=\frac{\dd+O(\dd n^{-0.4})}{2(n-O(n'))}=\frac{\dd}{2n}\Big(1+O(n^{-0.4})\Big).
\end{equation}  
Consequently, the probability that $\dddeg(s)=\ell$ conditioned on $\ddeg(s)=\dd$ and $\ddeg(s') = \dd'$
is given by 
$$\binom {n'}{\ell}\rho^\ell (1-\rho)^{n'-\ell}\Big(1+O(n^{-0.4})\Big)^{\ell}
\Big(1+O(n^{-0.4}\dd/n)\Big)^{n'-\ell}\,,$$
where $\rho=\dd/{2n}$.

If we additionally condition on the fact that $\dddeg(s')=\ddd'$ (so that we now have conditioned on $\ddeg(s) = \dd, \ddeg(s') = \dd'$, and $\dddeg(s') = \ddd'$), it will result in  
an extra factor of the order
$\Big(1+O(\dd/2n)\Big)^{\ddd'}$ since it means that some $\ddd'$
vertices already made their choice (and selected $s'$ as their neighbor). 
Note however that, since $\ell,\ddd'=O( \dd n'/n)=O(n^{0.11})$ we have
\begin{align*}
\Big(1+O(n^{-0.4})\Big)^{\ell}&=1+O(n^{-0.29})\\
\Big(1+O(n^{-0.4}\dd/n)\Big)^{n'-\ell}&=1+O(n^{-0.29})\\
\Big(1+O(\dd/2n)\Big)^{\ddd'}&=1+O(n^{-0.48})\,.
\end{align*} 
Hence, the probability that $\dddeg(s)=\ell$ conditioned on $\ddeg(s)=\dd$, $\ddeg(s')=\dd'$,  
and $\dddeg(s')=\ddd'$ is given by 
$$\binom {n'}{\ell}\rho^\ell (1-\rho)^{n'-\ell}\Big(1+O(n^{-0.29}))\,,$$
and so it is well approximated by the binomial distribution. On the other hand, 
the probability that the random variable with binomial distribution 
with parameters $n'$ and $\rho$ takes a particular value 
is bounded from above by  $O(1/\sqrt{n'\rho})$.
Thus, for a given pair of vertices $s<s'<k^2=n^{2\Delta}$  
we have 
$$P(\deg_n(s)=\deg_n(s'))=O(\sqrt{n/n'\dd})+O(n^{-c})=O(n^{-c})\,.$$  
Hence, the probability that such a pair of vertices, $s<s'<k^2=n^{2\Delta}$ 
exists is bounded from above by $O(k^4 n^{-c})$, and, as remarked at the beginning 
of the proof, $k = n^{\Delta}$ may be chosen small enough so that this yields a bound of the form $O(n^{-c'})$, for $c' > 0$.  
\end{proof}

\section{Proof of Theorem~\ref{th-main}}
\label{sec-proof1}

In this section we shall give a complete proof of Theorem~\ref{th-main}.
Let us first define two properties, $\fa$ and $\fb$ of $\PA(m; n)$ 
which are crucial for 
our argument. Here and below  we set $k = k(n) = n^{\Delta}$,
$\tk = \tk(n) = n^{\Delta'}$, and $\tk' = \tk'(n) = n^{\Delta''}$
for some small enough $0< \Delta < \Delta' < \Delta''$ to be chosen.  Specifically,
we will choose $\Delta''$ first, followed by $\Delta'$, followed by $\Delta$.

\begin{itemize}
	\item[($\fa$)] 
    	$\PA(m; n)$ has property $\fa$ if no two vertices $t_1,t_2$, where 	
        $k<t_1<t_2$, 
		are adjacent to the same $m$ neighbors from the set $[t_1-1]$.     	
	\item[($\fb$)] 
		$\PA(m; n)$ has property $\fb$ if the degree of every vertex  $s\le \tk$ 
		is unique in $\PA(m; n)$, i.e. for no other vertex $s'$ of $\PA(m; n)$ we have 
		$\deg_n(s)=\deg_n(s')$. 
\end{itemize}  

It is easy to see that 
\begin{equation}\label{eq000}
	P(|\Aut(\PA(m; n))|= 1)\ge P(\PA(m; n) \in \fa\cap \fb)\,,
\end{equation}
and so
\begin{equation}\label{eq001}
	P(|\Aut(\PA(m; n))|>1)
    \le P(\PA(m; n) \notin \fa) + P(\PA(m; n) \notin \fb)\,.
\end{equation}
Indeed, let us suppose that $\PA(m; n)$ has both properties $\fa$ and $\fb$, 
and  $\sigma\in \Aut(\PA(m; n))$. Let us assume also that $\sigma$ is not the identity, and let 
$t_1$ be the smallest vertex such that $t_2=\sigma(t_1)\neq t_1$. Note that $\fb$ implies that 
 for all $s\in [k]$ we have $\sigma(s)=s$, so that we must have 
$k<t_1<t_2$.  On the other hand  from $\fa$ it follows that $t_1$ and $t_2=\sigma(t_1)$ 
have  different 
neigbourhoods in the set $[t_1-1]$ (which consists of fixed points of $\sigma$, since $t_1$ was assumed to be the smallest non-fixed point of $\sigma$). This contradiction shows that $\sigma$ is the identity, i.e. $|\Aut(\PA(m; n))|=1$ which proves (\ref{eq000}).

Thus, in order to prove Theorem~\ref{th-main} it is enough to show that both 
probabilities $P(\PA(m; n) \notin \fa)$ and $P(\PA(m; n) \notin \fb)$ tend to $0$ 
polynomially fast as $n\to\infty$, for \emph{some} choice of $k, \tk$ (equivalently, $\Delta, \Delta'$).
We will show that $\Pr(\PA(m; n) \notin \fa)$ tends to $0$ polynomially fast for any choice of $\Delta$,
while $\Pr(\PA(m; n) \notin \fb)$ does so for some sufficiently small choice of $\Delta'$.

Let us study first the property $\fa$. Our task is to estimate from above 
the probability
that there exist vertices $t_1$ and $t_2$ such that $k<t_1<t_2$,  
which select the same $m$ neighbors (which, of course, belong to $[t_1-1]$).
Thus we conclude
\begin{align}
P(\PA(m; n)\notin \fa)
   	&\leq \sum_{k <t_1<t_2} P(t_1, t_2 
	  \text{ choose the same neighbors in } [t_1-1]) \nonumber \\
   	&\leq \sum_{k <t_1<t_2}
         \sum_{1 \leq r_1 \leq r_2 ... \leq r_m < t_1} P(t_1, t_2
			\text{ choose } r_1, ..., r_m)\,.
	\label{eq-rho}
\end{align}
The event in the last expression is an intersection of
dependent events, but it is known (see \cite{hofstaddiameter}, 
Lemma~2.1) that edge events involving connections to distinct 
target vertices (in this case, $r_1, ..., r_m$)
are negatively correlated, so that we can upper bound
by a product of probabilities:
\begin{align*}
	\Pr(t_1, t_2 \text{ choose } r_1, ..., r_m)
    \leq \prod_{\ell=1}^2 \prod_{s=1}^m \Pr(t_\ell \text{ chooses } r_s).
\end{align*}
Applying Lemma~\ref{AdjProbBoundCorollary}, for $k < t_1 < t_2$
we get
\begin{align*}
P(t_1, t_2
			\text{ choose } r_1, ..., r_m)
	&\le Cm	\prod_{\ell=1}^2\prod_{s=1}^m \frac{1}{\sqrt{t_\ell r_s}}.		
\end{align*}
Thus, (\ref{eq-rho}) becomes
\begin{align}
P(\PA(m; n)\notin \fa)
   	&\leq Cm\sum_{k <t_1<t_2}
         \sum_{1 \leq r_1 \leq r_2 ... \leq r_m < t_1} \prod_{\ell=1}^2\prod_{s=1}^m \frac{1}{\sqrt{t_\ell r_s}} \nonumber\\
 	&\leq Cm \sum_{k <t_1<t_2}(t_1t_2)^{-m/2}
         \sum_{1 \leq r_1 \leq r_2 ... \leq r_m < t_1} \prod_{s=1}^m \frac{1}{ r_s}
         \nonumber\\
	&\leq D(m)\sum_{k <t_1}t_1^{-m+1 + \alpha} 
         \nonumber\\  
         &\le C(m) k^{2-m+\alpha}, \label{TL1}    
\end{align}
where $C(m), D(m)$, and $\alpha$ are some positive constants 
(with $\alpha$ arbitrarily small and $C(m)$ and $D(m)$ depending
on $m$).  
The third inequality arises as follows: we have
\begin{align*}
 	\sum_{1 \leq r_1 \leq r_2 ... \leq r_m < t_1} \prod_{s=1}^m \frac{1}{ r_s}   
    &\leq \sum_{1 \leq r_1 \leq r_2 ... \leq r_{m-1} < t_1} \prod_{s=1}^{m-1} \frac{1}{r_s} \sum_{r_{m-1} \leq r_m < t_1} \frac{1}{r_m} \\
    &\leq  \sum_{1 \leq r_1 \leq r_2 ... \leq r_{m-1} < t_1} \prod_{s=1}^{m-1} \frac{1}{r_s} \log \frac{t_1}{r_{m-1}} \\
    &\leq \log t_1 \cdot \sum_{1 \leq r_1 \leq r_2 ... \leq r_{m-1} < t_1} \prod_{s=1}^{m-1} \frac{1}{r_s},
\end{align*}
so that, inductively,  this sum contributes a factor of $(\log t_1)^m \leq (\log t_2)^m$.  We may further upper
bound this by $D(m) t_2^{\alpha}$, for some arbitrarily small $\alpha$ and positive constant $D(m)$.    

Hence
\begin{equation}\label{eqa}
	P(\PA(m; n) \notin \fa)\le 
	n^{\Delta(2.0001-m)}\,,
\end{equation}
which is polynomially decaying since $m \geq 3$.
We remark that this holds for \emph{arbitrary} $\Delta > 0$.

Next we show that, with probability close to $1$, $\Delta'$ may be chosen sufficiently small
so that $\Pr(\PA(m; n) \notin \fb)$ tends to $0$ polynomially fast; i.e., the $\tk=n^{\Delta'}$  
oldest vertices  of $\PA(m; n)$ have unique degrees and so these are fixed points
of every automorphism. 
The key ingredient of our argument is Lemma~\ref{lkey}.

We first define some auxiliary properties: $\fb_1 = \fb_1(\Delta'')$ is the property that
the degrees of all vertices $< \tk'^2$ are pairwise different, and $\fb_2 = \fb_2(\Delta', \Delta'')$
is the property that the degrees of all vertices $< \tk$ are greater than those of all vertices
$> \tk'^2$.  It is easy to see that
\begin{align*}
	\Pr(\PA(m; n) \in \fb) 
    \geq \Pr(\PA(m; n) \in \fb_1, \PA(m; n) \in \fb_2),
\end{align*}
so that
\begin{align*}
	\Pr(\PA(m; n) \notin \fb)
    \leq \Pr(\PA(m; n) \notin \fb_1) + \Pr(\PA(m; n) \notin \fb_2).
\end{align*}

To estimate the probability that $\PA(m; n) \notin \fb_1$, we reason as follows:
from Lemma~\ref{lkey} we know that with probability at least 
$1 - O(n^{-c})$, for some positive constant $c$, 
the degrees of all vertices smaller than $\tk'^2 = n^{2\Delta''}$  
are pairwise different, for any $\Delta''$ small enough to satisfy Lemma~\ref{lkey}.  In other words,
for such a choice $\Delta''$, we have $\Pr(\PA(m; n) \notin \fb_1) = O(n^{-c})$.

Now we show that we can choose $\Delta'$ so as to upper bound $\Pr(\PA(m; n) \notin \fb_2)$.  Intuitively, we
will show that sufficiently early vertices have high degree, while sufficiently late vertices have low degree.
Using Corollary~\ref{EarliestDegreesCorollary}, 
one can deduce that we can choose $\Delta' > 0$ (playing the role of $\delta$ in the corollary) small enough so that, 
with probability at least $1-O(n^{-c})$ (for another positive constant $c > 0$) all vertices $s<\tk = n^{\Delta'}$ have degrees larger than those of
all vertices $t>\tk'^2$ (in particular using the left tail bound to show that vertices $< \tk$ all have high degree and the right tail bound to show that vertices $> \tk'^2$ have low degree \whp). 
Let us be more precise here: we can choose $\Delta$ in Corollary~\ref{EarliestDegreesCorollary} to be some very small
constant (say, $0.00001$).  This ensures the existence of a choice for $\Delta'$ for which
\begin{align}
	P\left( \intersect_{s < \tk = n^{\Delta'}} \deg_n(s) \geq Cn^{(1 - 0.00001)/2} \right) 
    \geq 1 - T^{-c},
\end{align}
for some positive constant $c$.  That is, with probability at least $1 - T^{-c}$, all vertices $< \tk$
have degree larger than $C n^{0.499995}$.  Note that we can bring the exponent arbitrarily close to $0.5$ by finding
a small enough $\Delta'$.  We will require, in fact, that $\Delta'$ is small enough compared to $\Delta''$.

Now, we can use Lemma~\ref{FriezeKaronskiLemma} to show that all vertices $r > \tk'^2 = n^{2\Delta''}$ have low
degree with high probability.  In particular, in the lemma statement, we choose $t := n$ and arbitrary $r > \tk'^2$.
This results, for arbitrary $A > 0$, in the bound
\begin{align}
	\Pr( \union_{r > \tk'^2} \deg_n(r) \geq Ae^m n^{(1 - 2\Delta'')/2}(\log n)^2)
    \leq n \cdot O(n^{-A}) = O(n^{1-A}),
\end{align}
which is polynomially decaying if we set $A$ sufficiently large (say, $A=2$).  
Note that we have used the union bound, followed by Lemma~\ref{FriezeKaronskiLemma}.  This shows that, with probability
at least $1 - n^{-c}$ for some $c > 0$, \emph{all} vertices $r > \tk'^2$ have degree at most $\tilde{O}(n^{(1 - 2\Delta'')/2})$.
Since $\Delta''$ is some specific (small) constant, we may choose $\Delta'$ above so small that
this is asymptotically smaller than the lower bound on the degrees of vertices $< \tk$.  So we have shown
that, for this choice of $\Delta'$, we have $\Pr(\PA(m; n) \notin \fb_2) = O(n^{-c})$.

Consequently, with probability $1-O(n^{-c})$ 
degrees of vertices from $[\tk]$ are unique; i.e.  $\PA(m; n) \notin \fb$.

Finally, Theorem~\ref{th-main} follows directly from 
(\ref{eq001}) and our estimates for
$\Pr( \PA(m; n) \notin \fa)$ and $\Pr( \PA(m; n) \notin \fb)$, provided that we choose $\Delta < \Delta'$.

\begin{remark}
	In the affine preferential attachment model with fixed
    $\delta > -m$, calculations give the following
    alterations of the above derivation.  Define
    $a = \frac{m}{2m+\delta}$.  Then Lemma~2.2 of 
    \cite{hofstaddiameter} implies that the probability
    that vertex $t_{\ell}$ chooses $r_s$ in any given choice
    is at most $\frac{1}{t_{\ell}^{1-a}r_s^{a}}$.
    Following the steps of the derivation above, we eventually
    get $k^{2 - \frac{2m(m+\delta)}{2m+\delta} + \alpha}$ in 
    place of
    $k^{2-m+\alpha}$ in the expression (\ref{TL1}).
    Setting the exponent equal to $0$ and solving for $m$ in
    terms of $\delta$ and $\alpha$ shows that we have
    asymmetry with high probability if 
    \begin{align*}
    	m > \frac{\sqrt{\delta^2 + 4} - \delta + 2}{2}.
    \end{align*}
    This recovers our result when $\delta = 0$, and it shows
    that the asymmetry threshold (as a function of $m$) occurs
    for smaller $m$ (at least $2$) as $\delta$ increases.
\end{remark}

\section{Proof of Theorem~\ref{StructuralEntropyTheorem}}
\label{StructuralEntropyTheoremProofSection}
We now prove the claimed estimate of the structural entropy.

We first show that the contribution of $\E[\log |\Aut(G)|]$ is negligible (in 
particular, $o(n)$).  From Theorem~\ref{th-main} and the fact that 
$|\Aut(G)| \leq n!$, we immediately have
\begin{align*}
	\E[\log|\Aut(G)|]
    \leq n\log n \cdot n^{-\delta} 
    = o(n).
\end{align*}

We now move on to estimate  $H(\sigma | \sigma(G))$, 
which we will show to satisfy
\begin{align}
	n\log n - O(n \log\log n)
    \leq H(\sigma | \sigma(G))   
    \leq n\log n - n + O(\log n).
    \label{ELogGamma}
\end{align}

To go further, we need to define a few sets which will play a role in our
derivation.
We define the \emph{admissible set}
$\Adm(S)$ of a given unlabeled graph $S$ to be the set of all labeled
graphs $g$ with $S(g) = S$ such that $g$ could have been generated
according to the preferential attachment model with given parameters.  
That is, denoting
by $g_t$ the subgraph of $g$ induced by the vertices $1, ..., t$
for each $t \in [n]$, we have that the degree of vertex $t$ in $g_t$
is exactly $m$.  We can similarly define $\Adm(g) = \Adm(S(g))$.  Then,
for a graph $g$, we define $\Gamma(g)$ to be the set of permutations
$\pi$ such that $\pi(g) \in \Adm(g)$.  We will also define, for an
arbitrary set of graphs $B$, 
\begin{align*}
	\Adm_B(g) = \Adm(g) \lintersect B, &&
    \Gamma_B(g) = \{ \pi ~:~ \pi(g) \in \Adm_B(g) \}.
\end{align*}

For a given graph $g$, these sets are related by the following formula 
(the simple proof of this fact is a tweak of that  given in Section~IIB of 
\cite{nodeage}):
\begin{align}
	|\Adm_B(g)| = \frac{|\Gamma_B(g)|}{|\Aut(g)|}.
    \label{AdmGammaAutEqn}
\end{align}

We next need to consider some directed graphs associated with $G$: we
start with $\dag(G)$, which is defined
on the same vertex set as $G$; there is an edge from $u$ to $v < u$ in $\dag(G)$ if
and only if there is an edge between $u$ and $v$ in $G$ (in other words, $\dag(G)$ is simply the graph $G$ 
before we remove edge directions).  Note that, if we ignore self-loops,
$\dag(G)$ is a directed, acyclic graph.

We denote the \emph{unlabeled} version of $\dag(G)$ (i.e., the set of
all labeled directed graphs with the same structure as $\dag(G)$) by
$\udag(G)$.  We will also, at times, abuse notation and write $\udag(G)$
as the set of all labeled, undirected graphs with the same structure
as $\udag(G)$ and with labeling consistent with $\udag(G)$ as a partial
order.

We have the following observations regarding these directed graphs.
\begin{lemma}
	\label{EquiprobabilityLemma}
    Let $G \distrib \PA(m; n)$ for any $m \geq 1$.
    For any two graphs $g_1, g_2$ satisfying $\udag(g_1) = \udag(g_2)$,
    we have
    \begin{align*}
    	P(G=g_1)
        = P(G=g_2).
    \end{align*}
\end{lemma}
\begin{proof}
    This can be seen by deriving a formula for the probability assigned
    to a given graph $g$ by the model and noting that it only depends on
    the structure and admissibility (a graph is said to be admissible if
    it is in $\Adm(S)$ for some unlabeled graph $S$).  If $g$ is not 
    admissible, then there exists some
    $t \in [n]$ such that the degree of vertex $t$ at time $t$ is not
    equal to $m$.  This has probability $0$, so $P(G=g) = 0$.
    
    Now, if $g$ is an admissible graph on $n$ vertices, then we can write $P(G=g)$ as
    a product over possible degrees of vertices at time $n$: let
    $\deg_g(v)$ denote the degree of vertex $v$ in $g$.  We consider
    the set $N_g(v)$ of immediate ancestors (i.e., the parents, the vertices that chose
    to connect to $v$) of $v$ in $\dag(g)$,
    denoting the number of edges that they supply to $v$ by 
    $d_1(v), ..., d_{k(v)}(v)$, where $k(v)$ is the number of parents of
    $v$.  
    Then we can write $P(G=g)$ as follows:
 	\begin{align*}
	    \Pr(G = g)
	    = \prod_{j=1}^{mn-1} (2mj)^{-1} \times (m!)^{n-1} \times 
	    \quad \prod_{d > m} 
	    \prod_{v ~:~ \deg_g(v)=d} \frac{d!}{(m-1)!} 
        \prod_{i=1}^{N_g(v)} \frac{1}{d_i!} 
	\end{align*}
   
    This arises from the following considerations: the probability $\Pr(G = g)$ is a ratio,
    the denominator of which is a product, with each vertex choice index $j$ contributing a 
    factor of $2m(j-1)$.  The numerator is computed as follows: the number of ways that
    each vertex's connection choices may have been made contributes a multinomial coefficient,
    resulting in the factorial factors of $(m!)^{n-1}$ and $\frac{1}{d_i!}$.  Finally, the
    fact that vertex $v$ in the $v$ product has degree $d$ (i.e., is chosen $d-m$ times)
    contributes the factor of $\frac{d!}{(m-1)!}$.
    
    Since this formula is only in terms of the degree sequence of the graph and $\udag(g)$,
    two graphs that are admissible and have the same unlabeled DAG 
    must have the same
    probability, which completes the proof.
\end{proof}
Strictly speaking, the above lemma only holds for the version of the model in which
vertex degrees are updated after every choice made by a new vertex.  However, the relative
rarity of multiple edges implies similar results (with additional factors of the form $e^{o(n)}$)
hold for other versions of the model.  This is sufficient for our entropy computations, since
we are generally concerned with expected logarithms of probabilities.

\begin{lemma}
	\label{NumDAGSUpperBound}
    Fix an unlabeled graph $S$ on $n$ nodes with $P(S(G) = S) > 0$ with
    some fixed $m \geq 1$.
    Then the number of distinct unlabeled directed graphs with undirected
    structure $S$ is at most $e^{\Theta(n)}$.
\end{lemma}
\begin{proof}
	Observe that the number of edges in $S$ is $\Theta(n)$, as it arises
    with positive probability from $\PA(m; n)$ and $m$ is fixed.
    Then note that each of the $\Theta(n)$ edges may be given one of two
    orientations, resulting in at most $2^{\Theta(n)}$ distinct directed
    graphs, which completes the proof.
\end{proof}

The next lemma shows that $H(\sigma | \sigma(G))$ may be expressed in
terms of the quantities just defined.
\begin{lemma}
	\label{RelatingToGamma}
    Fix $m \geq 1$ and consider $G \sim \PA(m; n)$.
    Let $\sigma \in \Sy_n$ be a uniformly random permutation.  Then 
    \begin{align}
    	H(\sigma | \sigma(G)) = \E[\log|\Gamma_{\udag(G)}(G)|] + O(n).
        \label{HCondSigAlternate}
    \end{align}
\end{lemma}
\begin{proof}
	First, we give an alternative representation of $H(\sigma | \sigma(G))$.
    Recall that $H(G | S(G)) = H(\sigma | \sigma(G)) - \E[\log|\Aut(G)|]$.
    The plan is to derive an alternative expression for $H(G | S(G))$ as
    follows: by the chain rule for entropy, we have
    \begin{align*}
    	H(G | S(G))
        &= H(G, \udag(G) | S(G)) \\
        &= H(\udag(G) | S(G)) + H(G | \udag(G)) \\
        &= O(n) + H(G | \udag(G)).
    \end{align*}
    Here, the last equality is a result of Lemma~\ref{NumDAGSUpperBound}.
    Now, by Lemma~\ref{EquiprobabilityLemma}, we have
    \begin{align*}
    	H(G | \udag(G))
        = \E[\log|\Adm_{\udag(G)}(G)|]
        = \E[\log|\Gamma_{\udag(G)}(G)|] - \E[\log|\Aut(G)|] + O(n),
    \end{align*}
    where the second equality is an application of (\ref{AdmGammaAutEqn}).
    This completes the proof.
\end{proof}

\begin{remark}
	Note that Lemma~\ref{RelatingToGamma} is robust to small variations
    in the model.
\end{remark}

Now, to calculate $H(\sigma | \sigma(G))$, it thus remains to estimate
$\E[\log|\Gamma_{\udag(G)}(G)|]$.

We will lower bound
$|\Gamma_{\udag(G)}(G)|$ in terms of the sizes of the \emph{levels} of $\dag(G)$, defined as follows:
$L_1$ consists of the vertices with in-degree $0$ (i.e., with total degree $m$).  Inductively, $L_j$ is the set of
vertices that are destinations \emph{only} of  
edges coming from vertices in $\lunion_{i=1}^{j-1} L_{i}$,
with at least one edge coming from $L_{j-1}$.  
Equivalently, a vertex
$w$ is an element of some level $\geq j+1$ if and only if there 
exist vertices $v_1 < \cdots < v_j$
such that $v_1 > w$ and the path $v_jv_{j-1} \cdots v_1 w$ exists 
in $G$. 

Then it is not too hard to see that any product of permutations that only permute vertices
within levels is a member of $\Gamma_{\udag(G)}(G)$.  Thus, we have, with probability $1$,
\begin{align*}
	|\Gamma_{\udag(G)}(G)|
    \geq \prod_{j \geq 1} |L_j|!.
\end{align*}

To continue, we will prove a proposition (Proposition~\ref{FatDAGProposition} below), 
to the effect that almost all vertices lie
in low levels of $\dag(G)$.  We define $X = X(\epsilon, k)$ to be the number of vertices
$w > \epsilon n$ that are at level $\geq k$ in $\dag(G)$.  In other words, $w$ is counted
in $X$ if there exist
vertices $v_1 < v_2 < \dots < v_k$ for which $w < v_1$ and the path $v_k\cdots v_1 w$ exists
in $\dag(G)$.  

We have the following lemma bounding $\E[X]$:
\begin{lemma}
	\label{EXLemma}
	For any $\epsilon = \epsilon(n) > 0$, there exists $k = k(\epsilon)$ for which 
    \begin{align*}
    	\E[X(\epsilon, k)] \leq \epsilon n.
    \end{align*}
    In particular, we can take any $k$ satisfying
    \begin{align}
    	k \geq 15\frac{m}{\epsilon^2}\log(3/\epsilon).
        \label{KEpsCondition}
    \end{align}
\end{lemma}
\begin{proof}
	Suppose that $w > \epsilon n$.  We want to upper bound the probability
    that there exist vertices $v_1 < \cdots < v_k$, with $w < v_1$, such
    that there is a path $v_k\cdots v_1 w$ in $G$.  
    Applying 
    Lemma~\ref{AdjProbBoundCorollary}, this probability is upper bounded by
    \begin{align*}
        {n \choose k} \cdot \left( \frac{Cm}{n\epsilon} \right)^{k} 
        \leq \left( \frac{Cme}{\epsilon k} \right)^k
    \end{align*}
    since 
    \begin{align*}
    	{n\choose k}n^{-k}
        \leq \frac{n^k}{k!}n^{-k}
        = \frac{1}{k!}
        \leq \frac{e^k}{k^k}.
    \end{align*}
    The last inequality is by considering the Taylor
    expansion of $e^k$ around $0$ and observing
    that it consists of strictly positive terms,
    including $\frac{k^k}{k!}$.
    
    Now, it is sufficient to show that we can choose $k$ so that this is $\leq \epsilon$. 
    In fact, we can choose 
    $k \geq \frac{3Cm}{\epsilon^2}$.
    This completes the proof.
\end{proof}
Now, we define $Y = Y(k)$ to be the number of vertices $w \geq 1$ that are at level $\geq k$ in 
$\dag(G)$.  The variables $X$ and $Y$ are related by the following inequalities, which hold with
probability $1$:
\begin{align*}
	X \leq Y \leq X + \epsilon n.
\end{align*}
Now, to get a bound on $Y$, we apply Markov's inequality:
\begin{align*}
	\Pr(Y \geq \gamma n)
    \leq \frac{\E[Y]}{\gamma n}
    \leq \frac{\E[X] + \epsilon n}{ \gamma n},
\end{align*}
and provided that (\ref{KEpsCondition}) holds, we can further bound by
\begin{align*}
	\Pr(Y \geq \gamma n) 
    \leq 2\epsilon / \gamma
\end{align*}
using Lemma~\ref{EXLemma}.
Then, provided that we choose $\gamma = \sqrt{2\epsilon}$, we have shown
that
\begin{align*}
	\Pr(Y \geq \gamma n) 
    \leq \gamma.
\end{align*}

This is summarized in the following proposition.
\begin{proposition}
	For any $\gamma = \gamma(n) > 0$, there exists $\ell = \ell(\gamma)$ for which
    the number of vertices that are not in the first $\ell$ layers of $\dag(G)$ is at
    most $\gamma n$, with high probability.
    In particular, we can take $\ell \geq 12 Cm / \gamma^4$.  
    \label{FatDAGProposition}
\end{proposition}

We have a final important result on the structure of $\dag(G)$.
The proof is given in the Appendix. 
\begin{theorem}[Height of $\dag(G)$]
	\label{DAGHeightUBTheorem}
	Consider $G_n \sim \PA(m; n)$ for fixed $m \geq 1$.  
    Then, with probability at least $1-o(n^{-1})$, the 
    height of $\dag(G_n)$ is at most $Cm\log n$, for some
    absolute positive constant $C$.
\end{theorem}

We now use Proposition~\ref{FatDAGProposition} to finish our lower bound on $\E[\log|\Gamma_{\udag(G)}(G)|]$.
Fix $\epsilon = \frac{1}{\log^2 n}$, so that $\gamma = \sqrt{2\epsilon} = \Theta(1/\log n)$,
and choose  $\ell = 12 Cm / \gamma^4$.  
Then, defining $A$ to be the event that the number of vertices in layers $> \ell$ is
at most $\gamma n = \Theta(n/\log n)$, 
we have
\begin{align*} 
	\E[\log |\Gamma_{\udag(G)}(G)|]
    \geq \E[\log |\Gamma_{\udag(G)}(G)| ~\big|~ A] (1 - \gamma).
\end{align*}
Among the $\ell$ layers, there are at most $\ell-1$ that satisfy, say, $|L_i| < \log\log n$,
since $\sum_{i=1}^{\ell} |L_i| \geq (1-\gamma)n$.  So we have the following:
\begin{align*}
	\sum_{i=1}^{\ell} \log(|L_i|!)
    = O(\ell \log\log n \log\log\log n) + \sum_{i \in B}(|L_i|\log|L_i| + O(|L_i|)),
\end{align*}
where $B = \{ i \leq \ell ~:~ |L_i| \geq \log\log n \}$, and we used Stirling's formula to
estimate the terms $i \in B$.  Note, importantly,
that the $O(|L_i|)$ term is uniform in $i$.

The sum $\sum_{i\in B} O(|L_i|) = O((1-\gamma)n) = O(n)$, so it remains to estimate
\begin{align*}
	\sum_{i \in B} |L_i|\log|L_i|.
\end{align*}
Let $N = \sum_{i \in B} |L_i|$.  Then, multiplying and dividing each instance of $|L_i|$ by
$N$ in the above expression, it becomes
\begin{align*}
	\sum_{i \in B} |L_i|\log|L_i|
    = N \sum_{i \in B} \frac{|L_i|}{N}\log \frac{|L_i|}{N} + N\sum_{i \in B} \frac{|L_i|}{N}\log N.
\end{align*}
The first sum is simply $-NH(X)$, where $X$ is a random variable distributed according to the empirical
distribution of the vertices on the levels $i \in B$.  Since $|B| \leq \ell$, we have that
$|-NH(X)| \leq N\log\ell$.  Thus, the first term in the above expression is $O(N\log\ell) = O(n\log\log n)$.
Meanwhile, the second term is $N\log N \sum_{i \in B} \frac{|L_i|}{N} = N\log N = n\log n - O(n\log\log n)$.
Thus, in total, we have shown
\begin{align*}
	\E[\log |\Gamma_{\udag(G)}(G)|]
    \geq n \log n - O(n\log\log n).
\end{align*}
Compare this with the trivial upper bound on $\E[\log|\Gamma_{\udag(G)}(G)|]$:
\begin{align*}
	\E[\log|\Gamma_{\udag(G)}(G)|]
    \leq \log n!
    = n\log n - n + O(\log n).
\end{align*}
This implies that we have recovered the first term of $\E[\log |\Gamma_{\udag(G)}(G)|]$, but there is a gap in our lower 
and upper bounds on the second term.  This completes the proof of (\ref{ELogGamma}).  Combining this with our estimates
of $\E[\log|\Aut(G)|]$ and of $H(G)$ yields the claimed structural entropy estimate, which concludes the proof of Theorem~\ref{StructuralEntropyTheorem}.

\section{Conclusion and Further Work}

In this paper, we just proved
that a version of the standard preferential attachment graph is asymmetric if every node
adds {\it more} than two edges. It is easy to extend this statement to the case
when the attachment is uniform and a mixture of uniform and preferential:
e.g., for a fixed $\beta \in [0, 1]$, the probability that a connection choice goes to node $w$ at time $n+1$ is
	$$
		P(v_i=w | G_{n}, v_1, ..., v_{i-1}) =\beta \frac{\deg_n(w)}{2m n} + (1-\beta) \frac{1}{n}.
	$$    
Another, possibly more practical, model was introduced by Cooper and Frieze \cite{cf2003}
in which essentially the number of edges added follows a given distribution.
We believe our methodology can handle this case, too.

However, consider a model in which the weight of a vertex when $m$ new edges are
generated is proportional to the degree raised to some power $\alpha$. In this
paper we considered $\alpha=1$. We are confident 
our approach could be adopted to work for all $\alpha >0$ to find
the threshold $m_\alpha$ for the asymmetry which, clearly, will grow with $\alpha$.
However, in the case $\alpha\neq 1$ the problem becomes much harder  since, for instance, 
the probability that $t$ chooses vertex $s$ as its neighbor  depends not 
only on the degree $\deg_t(s)$ but on the whole degree sequence at the time $t$ (though there has been some work on the asymptotic degree distribution and other structural properties
in the case of $\alpha > 1$ \cite{nonlinearm1,oliveira2005}).
Nonetheless, these difficulties could be overcome by modern combinatorial methods and we 
plan to deal with this model in the nearest future.

\bibliographystyle{plain}
\bibliography{asymmetry-bib.bib}

\section*{Appendix: Further Analysis of $\dag(G)$}
\label{HeightAppendix}

\begin{proof}[Proof of Theorem~\ref{DAGHeightUBTheorem}]
	Let us start with the following, surprising at first sight, observation.
	 
	\begin{fact}\label{f1}
	    Let $w < v$.  Then the degree $\deg_v(w)$ as well as the probability that $v$ is adjacent to $w$ does not depend on the structure of the graph induced by the first $w$ vertices.  In fact, the degree of a particular vertex is a Markov chain in
        time. \qed
	\end{fact}

	Let $p_m(n,k)$ denote the probability that $\dag(G_n)$ 
    contains a path of length $k$. 
    From Fact~\ref{f1} and 
    Lemma~\ref{AdjProbBoundCorollary}, it follows that 
	\begin{align}
		p_m(n,k)&\le \sum_{v_0<v_1<\dots<v_k} \prod_{i=1}^k\Pr(v_{i-1}\to v_i)
		\le \sum_{v_0<v_1<\dots<v_k} \prod_{i=1}^k\frac{5m\log (3v_i/v_{i-1})}{\sqrt{v_{i-1}v_i}}\nonumber\\
		&\le \sqrt{n} \sum_{v_0=1}^{n-k}\frac{1}{\sqrt{v_0}}\prod_{i=1}^k
		\sum_{v_i=v_{i-1}+1}^{n-k-i} \frac{5m\log (3v_i/v_{i-1})}{{v_i}}\,.\label{eq1}
	\end{align}
	In order to estimate the above sum we split all the vertices $v_1,\dots,v_k$ of the path $P$ into several classes. Namely we say that a vertex $v_i$ is of type $t$ in $P$ 
	if $t$ is the smallest natural number such that $v_i/v_{i-1}\le (1+a)^t$, where 
	$a$ is a small constant to be chosen later, i.e. $t=\lceil \log(v_i/v_{i-1})/\log (1+a)\rceil$. 
	Then, given $v_{i-1}$, the contribution of terms related to $v_i$ can be estimated from above by 
	\begin{equation}\label{eq2}
	\sum_{v_i=v_{i-1}(1+a)^{t-1}}^{v_{i-1}(1+a)^t}\frac{5m\log (3v_i/v_{i-1})}{{v_i}}\le 
	5m\log[(1+a)]\log[3(1+a)^t]\le \alpha t\,,
	\end{equation} 
	where, to simplify notation, we put $\alpha=5m\log(1+a)\log(3(1+a))$. Let 
	$s_t$ denote the number of vertices of type $t$ in $P$. Note that  
	$\prod_{t\ge 2}\big[(1+a)^{t-1}\big]^{s_t}\le n\,$ 
	and so 
	\begin{equation}\label{eq3}
	\sum_{t\ge 2} ts_t\le 2\sum_{t\ge 2} (t-1){s_t}\le \frac{2\log n}{\log(1+a)}\,.
	\end{equation}
	Let us set $J={2\log n}/{\log(1+a)}$. 
	Thus, we arrive at the following estimate for $p_m(n,k)$
	\begin{align*}
		p_m(n,k)&\le \sqrt{n} \sum_{v_0=1}^{n-k}\frac{1}{\sqrt{v_0}}\binom{k}{s_1} \alpha^{s_1}
		\sum_{\sum_t s_tt\le J}
        { {k-s_1}\choose {s_2, s_3 , ..., s_k} }
        \prod_{t\ge 2}^k (\alpha t)^{s_t}\\
		&\le 3n\binom{k}{s_1} \alpha^{s_1}
		\sum_{\sum_t s_tt\le J}
        { {k-s_1} \choose {s_2, s_3, ..., s_k}}
		 \exp\big(\sum_{t\ge 2}s_t\log(\alpha t)\big)\\
		&\le 3{n}\binom{k}{s_1} \alpha^{s_1}2^{2J}
		\max_{\sum_t s_tt\le J}
		 \exp\Big(\sum_{t\ge 2}s_t\log\Big(\frac{e\alpha t(k-s_1)}{s_t}\Big)\Big)\,.
	\end{align*}
	In order to estimate the expression
	$ \sigma(J,S)=\max_{\sum_t s_tt\le J}
	\exp\Big(\sum_{t\ge 2}s_t\log\Big(\frac{e\alpha t S}{s_t}\Big)\Big)
	$ %
	where ~~~ $S=\sum_{t\ge 2}s_t$, we split the set of all $t$'s into two parts.
	Thus, let 
	$T_1=\{t: \log (e\alpha tS/s_t)\le t\}$ and 
    $T_2=\{2,3,\dots,k\}\setminus T_1\,.$
	Then, clearly, 
	\begin{equation*}\label{eq5}
	\max_{\sum_t s_tt\le J}
	\exp\Big(\sum_{t\in T_1}s_t\log\Big(\frac{e\alpha t S}{s_t}\Big)\Big)
	\le \max_{\sum_t s_tt\le J}
	\exp\Big(\sum_{t\in T_1}s_t t\Big)\le \exp(J)\,.
	\end{equation*}
	Observe that for every $t\in T_2$ we have 
	$\log (eS\alpha t/s_t)\ge t$ and so $s_t\le e\alpha t  e^{-t}S$. It is easy to check that 
	then 
	$s_t\log\Big(\frac{e\alpha t S}{s_t}\Big)  \le 6\cdot 2^{-t}  S\,,    $
	so
	$$\max_{\sum_t s_tt\le J}
	\exp\Big(\sum_{t\in T_2}s_t\log\Big(\frac{e\alpha t S}{s_t}\Big)\Big)
	\le \max_{\sum_t s_tt\le J}
	\exp\Big(6S\sum_{t\in T_2}2^{-t}\Big)\le \exp(3S)\le \exp(3J)\,.$$
	Thus, 
	$\sigma(J,S)\le \exp(4J)\,,$
	and, since $s_1= k-S\ge k-J$, 
	\begin{align*}
	p_m(n,k)&\le  3{n}\binom{k}{s_1} \alpha^{s_1} 2^{2J}\sigma(J,k-s_1)
	\le 3 n 2^k \alpha^{k-J} \exp(6J)\\
	&\le 3\exp(\log n +k+  (k-J)\log \alpha +6J)\,.
	\end{align*}
	Since for $0<a<1$ we have 
	$a/2<\log (1+a)<a$, if we set $a=1/(310m)$, then $\alpha<1/61$ and $\log\alpha<-4$.
	Now let us recall that  $J={2\log n}/{\log(1+a)}$ and $k=5000m \log n>4J$.  Thus, 
	\begin{align*}
	p_m(n,k)&\le  3\exp(\log n +k+  (k-J)\log \alpha +6J)\\
	&\le  3\exp(\log n +k-  3 k+3k/2)=\exp(\log n - k/2)=o(n^{-1})\,.
	\end{align*}
\end{proof}
It is simple to show that with high probability
the height is also lower bounded by $\Omega(\log n)$.  Thus, the height is $\Theta(\log n)$ with high probability.


 
\end{document}